%% file: golomoziy_kladivko_mishura.tex
\documentclass{article}

\usepackage{arxiv}

\usepackage{graphicx} 
\usepackage{amsmath}
\usepackage{amsthm}
\usepackage{amssymb}
\usepackage{euscript}
\usepackage{dsfont}
\usepackage{bbm}

\newcommand{\lsum}{\sum\limits}

\newcommand{\prb}{\mathbb P}

\newcommand{\E}{\mathbb E}

\newcommand{\R}{\mathbb R}
\newcommand{\N}{\mathbb N}

\DeclareMathOperator{\Var}{Var}

\newtheorem{lem}{Lemma}
\newtheorem{thm}{Theorem}
\newtheorem{remark}{Remark}

\makeindex             

\begin{document}

\title{Discrete-time weak approximation of a Black-Scholes model with drift and volatility Markov switching}

\author{
  Vitaliy Golomoziy \\
    Department of Probability Theory and Mathematical Statistics,\\
    Taras Shevchenko National University of Kyiv,\\
    64 Volodymyrska, 01601, Kyiv, Ukraine,\\
    \texttt{vitaliy.golomoziy@knu.ua}\\
  \And
  \textbf{Kamil Klad\'{i}vko}\\
  School of Business,\\
  \"{O}rebro University, \\
    701 82 \"{O}rebro, Sweden,\\
    \texttt{kamil.kladivko@oru.se }\\
  \And
  \textbf{Yuliya Mishura}\\
    Department of Probability Theory and Mathematical Statistics,\\
    Taras Shevchenko National University of Kyiv,\\
    64 Volodymyrska, 01601, Kyiv, Ukraine,\\
    \texttt{yuliyamishura@knu.ua}
}

\date{2024-12-23}
%
%
\maketitle

\begin{abstract}
  We consider a continuous-time financial market with an asset whose price is modeled by a linear stochastic differential equation with drift and volatility switching  driven by a uniformly ergodic jump Markov process with a countable state space (in fact, this is a Black-Scholes model with Markov switching). We construct a multiplicative scheme of series of discrete-time markets with discrete-time Markov switching. First,  we establish that the  discrete-time switching Markov chains  weakly converge to the limit continuous-time Markov process. Second, having this in hand, we apply conditioning on Markov chains    and prove that the discrete-time market models themselves weakly converge to the Black-Scholes model with Markov switching.   The convergence is proved under very general assumptions both on the discrete-time net profits and on a generator of a continuous-time Markov switching process.
  \end{abstract}


\section{Introduction}
Without a doubt, real financial markets operate in discrete time. However, it is easier to perform analytical calculations on such markets, in particular, to evaluate contingent  claims, in continuous time. To eliminate this contradiction, it is logical to consider such models of financial markets with discrete time, which in one sense or another approximate well the models with continuous time, which, in turn, are amenable to analytical calculations. In particular, it is logical to consider stochastic models of discrete-time markets that weakly (i.e. in the sense of weak convergence of probability measures) approximate a suitable continuous-time model. This approach goes back to the weak convergence of the simplest Cox-Ross-Rubinstein model to geometric Brownian motion (the convergence of more general pre-limit schemes is contained in the book \cite{felmer}).
  
It is then natural to ask whether the convergence of prices of options and other contingent  claims is invariant with respect to such weak convergence of prices of primary assets. The answer to this question was considered in the works \cite{hubalek}, \cite{MiRa} and \cite{prigent}.
   Of course, these works do not exhaust the list of papers and books devoted to this issue.   

Since the introduction of the Black-Scholes model, it has become evident 
that the assumption of constant volatility lacks empirical support
in market option prices. Similarly, in portfolio optimization or 
trading strategies, modeling the rate of return as stochastic is
advantageous, as it aligns with observed patterns in time series of
asset returns. A popular, empirically supported, and relatively 
straightforward approach to incorporating stochastic volatility 
or drift into modeling is through Markov switching. This approach
assumes that the volatility and drift components 
of a geometric Brownian motion are functions of a Markov
process. Models of this nature are often referred to as 
\textit{regime-switching} or \textit{Markov-modulated} 
volatility and drift models.

An early study that introduced volatility switching into option
valuation is \cite{naik93}. Extensive research on option
pricing under Markov-modulated volatility followed. For 
developments in the continuous-time setting, see, 
for example, \cite{dimasi95} -- \cite{kirkby20}, listed in chronological order.

As an alternative to continuous-time models, various discrete-time lattice methods 
have been proposed. One of the early tree methods for 
modeling volatility switching between two states 
is the pentanomial tree introduced in \cite{bollen98}. 
For other tree methods that incorporate more than two 
volatility states or address the valuation of exotic options, 
see, for instance, \cite{aingworth06}-- \cite{yuen10}. 
These references, while far from exhaustive, are listed in 
chronological order. For discrete-time approximations 
of continuous-time Markov switching models in 
other applications, see, for example, 
\cite{savku_2017, Savku2021}.

Lattice methods are relatively easy to understand and implement, 
making them practically relevant in many situations. However, these 
methods are often constructed through intuitive discretization
of the continuous-time model, which does not guarantee 
convergence to the continuous-time counterpart.
To the best of our knowledge, relatively little attention 
has been paid to studying financial market models 
with discrete time and Markov switching that converge
to their continuous-time limits. Papers that present 
convergence results include \cite{guo01, yin04, leduc17}. Weak convergence is investigated in papers \cite{guo01, yin04}. However, our model is much more general, as we construct our discrete markets based on generic random variables that represent net profit rates produced by an asset, we consider potentially infinite set of switching states and impose more general conditions on the switching process generator. Moreover, we have successfully taken advantage of the benefits that the method of conditioning provides. Namely, applying classical results of \cite{gikhman-skorokhod-stoch-proc} on weak convergence of Markov processes, we establish weak convergence of our switching processes, and then we apply conditioning on Markov switching in order to prove weak convergence of the market model at whole. 
 
\section{Preliminaries}  So, our goal is to  construct a sequence of discrete-time financial markets that allow  drift and volatility switching by appropriate Markov chains (in paper \cite{driftSwitching} a similar approximation is constructed for a financial market with a drift switching only). This sequence is constructed in order to converge weakly to the Black-Scholes model with   switching by Markov process with continuous time.  The limit model has a form
\begin{equation}\label{initial_model}
    dX_t = \mu_{Y_t} X_t dt + \sigma_{Y_t} X_t d W_t,\ X_0 = x_0,
\end{equation}
where $W=(W_t)_{t\ge 0}$ is a Wiener process, $x_0>0$  is non-random, $\left(Y_t\right)_{t\ge 0}$ is a uniformly ergodic  Markov jump process with values in $\N$, independent of $W$.
We assume that the processes $(Y_t)_{t\ge 0}$ and $(W_t)_{t\ge 0}$ are defined on the 
stochastic basis with filtration $(\Omega, \EuScript{F}, (\EuScript{F}_t)_{t\ge 0}, \prb)$, 
where $\EuScript{F}_t = \sigma\{W_s, Y_s, 0\le s \le t\}$. 
Let $\{\mu_k, k\ge 1\}$ and $\{\sigma_k, k\ge 1\}$ be 
the sequences of real numbers, representing possible values for drift and volatility respectively, $\sigma_k > 0$ and 
\begin{gather}\label{sigma_def}
    \mu = \sup_{n\ge 1} |\mu_n| < \infty, \nonumber \\
    \sigma = \sup_{n\ge 1} \sigma_n < \infty.
\end{gather}
We assume that $Y_0=1$ and $(Y_t)_{t\ge 0}$ has an 
infinitesimal matrix
\begin{equation}\label{intensity_matrix}
  \mathbb{A} = \left(
  \begin{array}{cccc}
    -\lambda_1& \lambda_{12}& \lambda_{13}& \ldots \\
    \lambda_{21}& -\lambda_2& \lambda_{23}& \ldots \\
    \lambda_{31}& \lambda_{32}& -\lambda_3& \ldots \\
    \ldots
  \end{array}
  \right),
\end{equation}
such that $\lambda_{ij} > 0$ and $\sum_{j} \lambda_{ij} = \lambda_i$, for all $i \in \N$.
We assume that 
\begin{equation}\label{lambda_start_def}
\lambda_* = \inf_i \lambda_i > 0,\ \mbox{and}\ \lambda^* = \sup_i \lambda_i < \infty,
\end{equation}
Let us define jump and occupation times for $\left(Y_t\right)_{t\ge 0}$ as follows:
\begin{gather}\label{def_tau}
\tau_0 = 0,\ \tau_n = \inf\{t \ge \tau_{n-1} : Y_t \neq Y_{\tau_{n-1}}\},\ n\ge 1,\\
\theta_0 = 0,\ \theta_k = \tau_{k} - \tau_{k-1},\ k\ge 1. \nonumber
\end{gather}
Also, the number of jumps equals
\begin{equation}\label{Nt_def}
N_t = \sup\{n \ge 0: \tau_n < t\},
\end{equation}
and embedded Markov chain $\{y_n, n\ge 0\}$ has a form
\begin{equation}\label{emb_chain_def}
y_n = Y_{\tau_n}, n\ge 0.
\end{equation}

The unique strong solution of equation  \eqref{initial_model}  can be represented as 
an exponent of the form
\begin{equation*} 
X_t = X_0 \exp\left(\lsum_{k=0}^{N_t} \left(\left(\mu_{y_k} - {\sigma^2_{y_k}\over 2}\right)\left(\tau^t_{k+1} - \tau_k\right) + \sigma_{y_k}\left(W_{\tau^t_{k+1}} - W_{\tau_k}\right)\right)
\right),
\end{equation*}
where $\tau^t_k = \tau_k \wedge t$, $N_t$ is defined in (\ref{Nt_def}) and $\tau_k$, $y_k$ are defined in (\ref{def_tau}), (\ref{emb_chain_def}) correspondingly. 
We will find it convenient to work with the logarithms
\begin{equation}\label{ut_def}
U_t = \log\left(X_t\right) = \log\left(X_0\right) + \lsum_{k=0}^{N_t} \left(\left(\mu_{y_k} - {\sigma^2_{y_k}\over 2}\right)\left(\tau^t_{k+1} - \tau_k\right) + \sigma_{y_k}\left(W_{\tau^t_{k+1}} - W_{\tau_k}\right)\right).
\end{equation}

\section{Discrete-time multiplicative approximation of the diffusion model 
with Markov switching}\label{discrete-time-definitions}

The main goal of this paper is to construct a sequence
of discrete-time versions of $X$, that is the geometric Brownian motion
with Markov modulated drift and volatility given by~\eqref{initial_model}, such
that these discrete-time versions weakly converge in Skorokhod topology 
 to the process $X$ on any fixed time interval $[0,T]$.  

So, moving in this direction, we consider the limit process $X$ on the fixed time interval $ [0,T]$, where $T>0$ is a maturity date, and create a scheme of series containing  discrete-time models, numbered by $N\in \mathbb{N}$ . Our $N$th discrete-time market corresponds to the partition of the interval $[0,T]$ into $N$ subintervals of the form $\left[{(k-1)T\over N}, {kT\over N}\right]$, $1\le k \le N$. 
Let $X^{(N)}_0=x_0$, and $X^{(N)}_k$ be a strictly positive discounted price of the asset at a time ${kT\over N}$ of $N$th discrete-time market, $1\le k \le N$.

Taking into account the multiplicative nature of the limit model, together with the assumption of independence of $Y$ and $W$ on $[0,T]$, we can assume that the ratio $X^{(N)}_{k}\over {X^{(N)}_{k-1}}$, $ 1\le k \le N$ can be represented as 
\begin{equation}\label{ratios}
     {X^{(N)}_{k} \over X^{(N)}_{k-1}} = \left(1+R^{(N)}_{k, Y^{(N)}_k}\right), \end{equation}
where the random variables $R^{(N)}_{k,j}$, $1\le k \le N$, $j\in \N$ are independent, $R^{(N)}_{k,j} > -1$ a.s., and $\left(Y^{(N)}_k\right)_{k\ge 0}$ is a discrete-time switching process independent of all $\left\{R^{(N)}_{k,j}, 1\le k\le N, j\in \N\right\}$.
Here $R^{(N)}_{k,j}$ represent a net profit rates generated by the price process on the time intervals $\left[{(k-1)T\over N},{kT\over N}\right]$, $1\le k \le N$ in a model with the switching.
Taking logarithms, we can write
\begin{equation}\label{prod_representation}
 U^{(N)}_k = \log(X^{(N)}_{k}) = \log(X_0) + \lsum_{j=1}^{k} \log\left(1 + R^{(N)}_{j, Y^{(N)}_j}\right), 
\end{equation}
where $ 1\le k \le N$ and $U^{(N)}_0 = \log(x_0)$.
We will use processes $\left(U^{(N)}_k\right)_{0\le k\le N}$  to approximate the process $\left(U_t\right)_{t\in [0,T]}$ defined in (\ref{ut_def}), and will find it convenient to define such discrete processes on a continuous time interval $[0,T]$ rather than a discrete set $\{0,1,\ldots, N\}$. To this end, we define a continuous-time processes
\begin{gather*}
Y^{(N)}_t = Y^{(N)}_{\left\lfloor {tN\over T} \right\rfloor}, U^{(N)}_t = U^{(N)}_{\left\lfloor {tN\over T} \right\rfloor},   X^{(N)}_t = X^{(N)}_{\left\lfloor {tN\over T} \right\rfloor}. \nonumber
\end{gather*}
Here, by $\lfloor a \rfloor$ we understand the biggest integer smaller or equal to $a$.
We assume that the process $\left(X^{(N)}_t\right)_{t\ge 0}$, and thus random variables $\left\{R^{(N)}_{i,j}, 1\le i\le N, j\in \N\right\}$ and $\left(Y^{(N)}_k\right)_{k\ge 0}$ are defined on a stochastic basis $\left(\Omega^{(N)}, \EuScript{F}^{(N)}, \left(\EuScript{F}^{(N)}_t\right)_{t\in [0,T]}, \prb^{(N)}\right)$,
where filtration is generated by the respective random variables $R^{(N)}_{k,j}, Y^{(N)}_k$
$$ \EuScript{F}^{(N)}_t = \sigma\left\{Y^{(N)}_k, R^{(N)}_{k,j}, 1\le k\le \left\lfloor{tN\over T}\right\rfloor,\ j\in \N\right\}, $$
so that  $X^{(N)}_t$, $Y^{(N)}_t$ and $U^{(N)}_t$ are $\EuScript{F}^{(N)}_t$-measurable.

Recall that the process $(Y_s)_{s\ge 0}$  is  the jump Markov process
with values in the set $\N$ and with infinitesimal matrix (\ref{intensity_matrix}). This process   governs the switching in continuous-time model.
Recall also, that state $i$ provides drift rate $\mu_i$. 
Once we consider a discrete-time model, we have to introduce a discrete-time switching process (note that in such a model the switching of the drift rate may only occur at a times $kT\over N$) which is the process $\left(Y^{(N)}_k\right)_{ k\ge 0}$ introduced above. This Markov chain takes values in the set $\N$, has initial value $Y^{(N)}_0=1$ and $Y^{(N)}_k=j$ implies that the intensity of the interest on $k$th interval of $N$th discrete-time market equals   $\mu_j$. 

The definition of the transition probability for the process $\left(Y^{(N)}_k\right)_{k\ge 0}$ follows from the requirement for 
occupation times of $\left(Y^{(N)}_k\right)_{k\ge 0}$ to be close  to that of $(Y_s)_{s\ge 0}$.  This leads to the following definition of transition probabilities of the chain $\left(Y^{(N)}_k\right)_{k\ge 0}$
for $i,j \in \N$:
\begin{align*}
\prb^{(N)}&\left(Y^{(N)}_{k+1} = i\ \middle|\ Y^{(N)}_k = i\right)=\prb\left(Y_s = i, {{Tk}\over N} \le s \le {{T(k+1)}\over N}\ \middle|\ Y_{Tk\over N} =i\right) \\
 &= \prb(Y_s = i, 0\le s \le T/N\ |\ Y_0 = i) = e^{- {\lambda_i T\over N}},
\end{align*}
\begin{align*}
p^{(N)}_{ij} = \prb^{(N)}&\left(Y^{(N)}_{k+1} = j\ \middle|\ Y^{(N)}_k = i\right) =\prb\left(Y_{T(k+1)\over N} = j\ \middle|\ Y_{Tk\over N} =i\right) \\
 &= \prb\left(Y_{T\over N} = j\ \middle|\ Y_0 = i\right) = p_{ij}\left({T\over N}\right),
\end{align*}
where we used the Markov property of the process $(Y_s)_{s\ge 0}$ and denoted the solution of Kolmogorov equations for $(Y_s)_{s\ge 0}$ by $p_{ij}(t)$. 
Such probabilities $p^{(N)}_{ij}$ define one-step transition probabilities matrix
\begin{equation}\label{yn_trans_probability}
    \left(
    \begin{array}{cccc}
    e^{-\lambda_1 {T\over N}}& p^{(N)}_{12}& p^{(N)}_{13}& \ldots\\
    p^{(N)}_{21}&  e^{-\lambda_2 {T\over N}}& p^{(N)}_{23}& \ldots\\
    \ldots
    \end{array}
    \right).
\end{equation}

\section{Weak convergence of discrete-time Markov chains to the limit Markov process}\label{markov_conv_section}
In this section, we prove that the sequence of processes $\left(Y^{(N)}\right)_{t\in [0,T]}$ introduced in Section \ref{discrete-time-definitions} converges in Skorokhod topology to the process $\left(Y_t\right)_{t\in [0,T]}$. 

Recall that Markov chain $\left(Y^{(N)}_k\right)_{k\ge 0}$,   introduced in Section \ref{discrete-time-definitions}, has initial value $Y^{(N)}_0=1$ and transition probability matrix (\ref{yn_trans_probability}). For this chain let us define the jump times
$$ \tau^{(N)}_0 = 0,\ \tau^{(N)}_n = \inf\left\{k > \tau^{(N)}_{n-1} : Y^{(N)}_k \neq Y^{(N)}_{\tau^{(N)}_{n-1}}\right\}, n \ge 1,$$
and occupation times
\begin{equation}\label{tauNk_def} 
\theta^{(N)}_0 = 0,\ \theta^{(N)}_k = \tau^{(N)}_k - \tau^{(N)}_{k-1}, k\ge 1.
\end{equation}
For a given $t\in [0,T]$ and integer $N$ define $k_{t,N}\in \{0, \ldots, N\}$ in the following way:  $k_{T,N} = N$, and for $t\in [0,T)$ we have $t\in \left[{k_{t,N} T\over N}, {(k_{t, N}+1)T\over N}\right).$ We will also use a notation
$t^{(N)} = {k_{t,N}T\over N}$.
Note, that we will use the symbol $N_T$ (previously introduced in (\ref{Nt_def})) to define the total number of jumps of the process $\left(Y^{(N)}_t\right)_{t\in [0,T]}$. It will be clear from the context if we mean the total number of jumps of the process $\left(Y_t\right)_{t\in [0,T]}$ or of the process $\left(Y^{(N)}_t\right)_{t\in [0,T]}$.

\begin{lem}\label{nt_exp_moment}
For every $\alpha > 0$
$$ \E e^{\alpha N_T} \le e^{-\lambda_* T} + \Lambda \exp\left(\alpha + e^\alpha\lambda^* T\right), $$
where
$ \lambda^* = \sup_i \lambda_i,\ \lambda_* = \inf_i \lambda_i$ are defined in (\ref{lambda_start_def}) and $\Lambda = {\lambda^* \over \lambda_*}.$
\end{lem}
\begin{proof}
Note that given $y_1=j_1,\ldots, y_m=j_m$, random variable $\tau_m$ is a sum of $m$ independent random variables $\{\theta_1, \ldots, \theta_m\}$ (both defined in (\ref{def_tau})), such that $\theta_i$ has an exponential distribution with parameter $\lambda_{j_i}$, so that $\tau_m$ has a hypoexponential distribution. Taking into account an estimate for the exponential densities:
$$ \lambda_i e^{-\lambda_i x} \le \left(\lambda_i\over \lambda_*\right)\lambda_* e^{-\lambda_* x} \le \Lambda \lambda_* e^{-\lambda_* x},\ x\ge 0, $$
we can estimate hypoexponential distribution function of $\tau_m$ by an Erlang distribution function with parameter $\lambda_*$ in the following way:
$$ F_{\tau_m}(t) = \prb\left(\tau_m < t\ |\ y_1=j_1,\ldots,y_m=j_m\right) \le \Lambda^m F_{\lambda_*, m}(t),$$
where 
 $F_{\lambda,m}(x) $ is a distribution function for Erlang distribution with parameters $\lambda, m$. It is well known that
$$ F_{\lambda, m}(x)  = {\int_0^{\lambda x} t^{m-1} e^{-t}dt \over (m-1)!}\le {(\lambda x)^{m-1}\over (m-1)!}.$$
We can now write
\begin{align*}
    \prb(N_T = m) &= \prb(\tau_m \le T < \tau_{m+1}) \le \prb(\tau_m \le T) \\
    &= \lsum_{j_1,\ldots,j_m} \prb(\tau_m \le T\ |\ y_1= j_1,\ldots, y_m= j_m) \prb(y_1 = j_1, \ldots, y_m=j_m) \\
    &\le \lsum_{j_1,\ldots,\_m} 
    \Lambda^m F_{\lambda_*, m}(T) \prb(y_1 = j_1,\ldots,y_m=j_m)\\
    &= \Lambda^m F_{\lambda_*, m}(T).
\end{align*}
Therefore
\begin{align*}
 \mathbb{E} e^{\alpha N_T} &\le e^{-\lambda_* T} +  \lsum_{m\ge 1} e^{\alpha m}\Lambda^m {\left(\lambda_* T \right)^{m-1} \over (m-1)!} = e^{-\lambda_* T} + \Lambda e^{\alpha} \lsum_{m\ge 0} {(e^{\alpha}\lambda^* T)^m \over m!} \\
&= e^{-\lambda_* T} + \Lambda \exp\left(\alpha + e^\alpha\lambda^* T\right). 
\end{align*}
\end{proof}
\begin{remark}
Note that the bound
$$ \sup_N \mathbb{E}^{(N)} e^{\alpha N_T} < \infty,$$
can be established in a similar way.
\end{remark}

\begin{thm}\label{approx_y_thm1}
Denote by $f_m(t_1, \ldots, t_m)$ a conditional density of $(\tau_1, \ldots, \tau_m)$ (defined in (\ref{def_tau})) given $N_T = m$ and put
$$ g_m(t_1, \ldots, t_m) =f_m(t_1, \ldots, t_m) \prb(N_T=m). $$
Then the following equation holds true for all $0\le t_1 < \ldots < t_m \le T$:
$$ \lim_{N\to \infty} \left({N\over T}\right)^{m+1} \hat{p}_N(k_{t_1, N}, \ldots, k_{t_m, N}) = g_m(t_1, \ldots, t_m), $$
where 
$\hat{p}_N(k_{t_1,N}, \ldots, k_{t_m, N}) =\prb^{(N)}\left\{ \tau^{(N)}_1 = k_{t_1,N}, \ldots, \tau^{(N)}_m = k_{t_m,N},\ \tau^{(N)}_{m+1}>N\right\}$,\\
and $\{\tau^{(N)}_k, k\ge 1\}$ are defined in  (\ref{tauNk_def}).
\end{thm}
\begin{proof} 
Let $\tilde f_{j_1,\ldots,j_m}(u_1, \ldots, u_{m})$ be a conditional density of $(\theta_1,\ldots, \theta_{m})$ given $N_T=m$ and $y_i=j_i,\ 1\le i\le m$ (where $y_i$ defined in (\ref{emb_chain_def})),
so that
\begin{equation}
f_m(t_1,\ldots, t_m) = \mathbb{E}[\tilde f_{y_1,\ldots, y_m}(t_1, t_2-t_1,\ldots, t_m-t_{m-1})],
\end{equation}
and let us introduce
$$\tilde g_{j_1,\ldots,j_m}(u_1,\ldots, u_{m}) = \tilde f_{j_1,\ldots,j_m}(u_1, \ldots, u_{m}) \prb(N_T =m). $$

 Since $\theta_1, \theta_2,\ldots, \theta_{m}$ are independent random variables with exponential distributions we can write
\begin{align*}
\tilde g_{j_1,\ldots,j_m}(u_1, \ldots, u_{m}) &= \lim_{h\to 0} {1\over (2h)^{m}} \prb\left(|\theta_j - u_j| < h, N_T = m,\ 1\le j\le m\right) \\
&= \prb(\theta_{m+1} > T-(u_1+\ldots + u_{m})) \lim_{h\to 0}  \prod\limits_{j=1}^{m} \left({1\over 2h} \prb(|\theta_j-u_j| < h)\right) \\
&= \prod\limits_{i=0}^{m-1} \left(\lambda_{j_i} e^{-\lambda_{j_i} u_{i+1}}\right) e^{-\lambda_{j_m} (T-(u_1 + \ldots + u_{m}))},
\end{align*}
where $j_0=1$ and $u_j \ge 0$, $1\le j \le m$ such that $u_1 + \ldots + u_{m} \le T$.

Recall that $\tau_0=0$ and $\theta_j = \tau_{j} - \tau_{j-1}$, $1\le j \le m$. So, for all $0 = t_0 \le t_1 < \ldots < t_{m} \le t_{m+1} = T$
\begin{align*}
g_{m}&(t_1, \ldots, t_{m}) = \mathbb{E}\tilde g_{y_1,\ldots,y_m}(t_1, t_2-t_1, \ldots, t_{m} - t_{m-1}) \\
&= \mathbb{E}\prod\limits_{i=0}^{m-1} \left(\lambda_{y_i} e^{-\lambda_{y_i} (t_{i+1}-t_i)}\right) e^{-\lambda_{y_m} (T-t_m)}\\
&= \mathbb{E}\left[\left(\prod\limits_{i=0}^{m-1} \lambda_{y_i}\right) \exp\left(-\lsum_{i=0}^{m} \lambda_{y_i} (t_{i+1}-t_i) \right) \right].
\end{align*}

To simplify the further derivations, let us omit indices in $k_{t_i, N}$ and simply write $k_i$.
Then we can rewrite $\hat{p}_N(k_1, \ldots, k_m)$ as 
\begin{gather*}
    \hat{p}_N(k_1, \ldots, k_{m}) = \lsum_{j_1,\ldots, j_m}
    \left(\prod\limits_{i=0}^{m-1} e^{-\lambda_{j_i} {(k_{i+1} - k_i)T\over N}}p^{(N)}_{j_i j_{i+1}}\right) e^{-\lambda_{j_m}\left(T - k_m {T\over N}\right)},
\end{gather*}
where $j_0 = Y_0 = 1$ is a fixed initial value.
From \cite{chung}, Chapter II.3 (p.130), and Theorem 1 at the same page, we know that
\begin{equation}\label{pin_maj}
\lim_{N\to \infty} \left({N\over T}\right)p^{(N)}_{ij} = \lambda_{ij},\ \mbox{and}\ p^{(N)}_{ij} \le \lambda_i\left({T\over N}\right),\ i\neq j.
\end{equation}
Thus, we have
\begin{equation}\label{dom_conv}
    \left({N\over T}\right)^m\hat{p}_N(k_1, \ldots, k_{m}) = \lsum_{j_1,\ldots, j_m}\left(\prod\limits_{i=0}^{m-1} e^{-\lambda_{j_i} {(k_{i+1} - k_i)T\over N}} p^{(N)}_{j_i j_{i+1}} {N\over T}\right) e^{-\lambda_{j_m}\left(T - k_m {T\over N}\right)},
\end{equation}
where $\lsum_{j_1,\ldots,j_m}$ should be understood as 
$\lsum_{\substack{j_1 \in \N\\ j_1\neq 1}}\lsum_{\substack{j_2 \in \N\\ j_1\neq j_1}}\ldots \lsum_{\substack{j_m \in \N\\ j_m\neq j_{m-1}}}.$
Using Lemma \ref{unif_conv_mark_chain} we can go to the limit 
\begin{align*}
    \lim_{N\to \infty} &\left({N\over T}\right)^m\hat p_N(k_1, \ldots, k_m) \\
    &= \lsum_{j_1,\ldots,j_m}\left[ \exp\left(-\lsum_{i=0}^{m-1} \lambda_{j_i} (t_{i+1} - t_i)\right)\prod\limits_{i=0}^{m-1} \lambda_{j_i j_{i+1}}\right]e^{-\lambda_{j_m}(T-t_m)}\\
    &= \lsum_{j_1,\ldots, j_m}\left[ \exp\left(-\lsum_{i=0}^{m-1} \lambda_{j_i} (t_{i+1} - t_i)\right)\left(\prod\limits_{i=0}^{m-1} {\lambda_{j_i, j_{i+1}}\over \lambda_{j_i}}\right) \prod\limits_{i=0}^{m-1} \lambda_{j_i}\right] e^{-\lambda_{j_m}(T-t_m)} \\
    &= \E\left[\exp\left(-\lsum_{i=0}^{m} \lambda_{y_i} (t_{i+1} - t_i)  \right)\prod\limits_{i=0}^{m-1} \lambda_{y_i} \right] = g_m(t_1, \ldots, t_m),
\end{align*}
where $t_0=0, t_{m+1}=T$ as before.
Theorem is proved.
\end{proof}

\begin{thm}\label{approx_y_thm2}
  Let $0 \le t_0 < t_1 < \ldots < t_n \le T$ be fixed. Then for any $\varepsilon > 0$ there exists such $N(n, \varepsilon)$ that for all $N\ge N(n, \varepsilon)$ we have
  \begin{equation}\label{finite_dim_distributions}
    \left| \prb\left(Y_{t_i} = x_i, 0\le i\le n\right) - \prb^{(N)}\left(Y^{(N)}_{t^{(N)}_i} = x_i, 0\le i\le n\right)\right| < \varepsilon,
  \end{equation}
  where $x_i \in \N$, $0\le i \le n$.
\end{thm}
\begin{proof}
For every fixed $\varepsilon>0$
we can find $m$ such that $\prb(N_T > m) + \prb^{(N)}(N_T > m) < \varepsilon/2$. Then (\ref{finite_dim_distributions}) is reduced to
  \begin{equation}\label{finite_dim_distributions_1}
    \left| \prb\left(Y_{t_i} = x_i, 0\le i \le n, N_T \le m\right) - \prb^{(N)}\left(Y^{(N)}_{t^{(N)}_k} = x_i, 0\le i \le n, N_T\le m\right)\right| < \varepsilon/2.
  \end{equation}

  Let us introduce the  random variables $r_j$ such that
  \begin{equation}
     \tau_{r_j} \le t_j < \tau_{r_j+1}.
  \end{equation}
  In fact,  $r_j$ is the number of the occupation interval that covers the fixed point $t_j$. Note  that $r_j$ is defined on the same probability space as $(Y_s)_{s\ge 0}$, and for $\omega \in \{N_T \le m\}$ each $r_k(\omega)$ takes
  value in the set $\{0,1,\ldots, m\}$. Put 
  $$A^m_n = \{(r_0(\omega), \ldots, r_n(\omega)),\, \omega \in \{N_T < m\}\} \subset \mathbb{R}^{n+1}. $$
  It is clear that $A^m_n$ is a finite set, and $ |A^m_n| \le m^{n+1}. $\\
For every $r=(r_1,\ldots, r_n) \in A^m_n$ let us define
  \begin{gather}\label{br_def}
   B_r = \{\tau_{r_j} \le t_j < \tau_{r_j+1}, 0\le j\le n, N_t\le m \},\\
   B^{(N)}_r = \{\tau^{(N)}_{r_j} \le t_j < \tau^{(N)}_{r_j+1}, 0\le j\le n, N_t\le m \}. \nonumber
\end{gather} 
Then we get the following equalities 
$$ \prb\left(Y_{t_i} = x_i, 0\le i\le n, N_T\le m\right) =\lsum_{r\in A^m_n} \prb\left(Y_{\tau_i} = x_i, 0\le i\le n, B_r\right), $$
and
$$ \prb^{(N)}\left(Y^{(N)}_{t^{(N)}_i} = x_i, 0\le i\le n, N_T\le m\right) =\lsum_{r\in A^m_n} \prb^{(N)}\left(Y^{(N)}_{\tau^{(N)}_i} = x_i, 0\le i\le n, B^{(N)}_r\right). $$  
Let us show, that for any fixed $r\in A^m_n$ the following convergence hold
\begin{equation}\label{ytaun_conv}
\lim_{N\to \infty} \prb^{(N)}\left(Y^{(N)}_{\tau^{(N)}_i} = x_i, 0\le i\le n, B^{(N)}_r\right) = \prb\left(Y_{\tau_i} = x_i, 0\le i\le n, B_r\right).
\end{equation}
Since $A^m_n$ is a finite set, this will prove (\ref{finite_dim_distributions_1}) and the statement of the theorem follows.
We will show (\ref{ytaun_conv}) by induction in $n$. We start with the case $n=1$. Let $t\in [0,T]$, $r\in\{1,\ldots, m\}$ and $j_0,j_r\in \N$ be some fixed numbers. Put $s_0=0$, and assume that $\lsum_i^j = 0,\ i>j$ by convention.
Then we can write
\begin{align}\label{fin_dim_lim}
 \lim_{N\to \infty} &\mathbb{P}^{(N)}\left(Y^{(N)}_{\tau^{(N)}_r} = j_r, \tau^{(N)}_r \le t < \tau^{(N)}_{r+1}\ \middle|\ Y^{(N)}_0=j_0\right) \nonumber \\
 &= \lim_{N\to \infty}\lsum_{s_1=1}^{k_{t,N}}\lsum_{s_2=s_1 + 1}^{k_{t,N}}\ldots \lsum_{s_r=s_{r-1}+1}^{k_{t,N}} \lsum_{j_1,\ldots, j_{r-1}} \left(\prod\limits_{k=0}^{r-1} e^{-\lambda_{j_k} (s_{k+1}-s_{k}-1){T\over N}}p^{(N)}_{j_k j_{k+1}}\right)e^{-\lambda_{j_r}(k_{t,N}-s_r){T\over N}} \nonumber \\
 &= \lim_{N\to \infty}\int_0^t \int_{u_1}^t \ldots \int_{u_{r-1}}^t  \lsum_{j_1,\ldots, j_{r-1}} {f^{(N)}_{j_0,\ldots,j_r}(\bar u)\over (T/N)^r} \prod\limits_{k=0}^{r-1} p^{(N)}_{j_{k-1}j_k}d \bar u \nonumber \\
&= \lim_{N\to \infty}\int_0^t \int_{u_1}^t \ldots \int_{u_{r-1}}^t  \lsum_{j_1,\ldots, j_{r-1}} f^{(N)}_{j_0,\ldots,j_r}(\bar u) \prod\limits_{k=0}^{r-1} \left(p^{(N)}_{j_{k-1}j_k} {N\over T}\right)d \bar u, 
  \end{align}
where $\bar u = (u_1,\ldots, u_r)$, and $f^{(N)}_{j_0,\ldots,j_r}(\bar u)$ is a function, that is constant on every hypercube $C^{(N)}(s_1,\ldots, s_r) =  \left[{(s_1-1) T\over N}, {s_1T\over N}\right) \times \ldots \times \left[{(s_r-1) T\over N}, {s_rT\over N}\right)$, where $1\le s_1 < s_2 < \ldots < s_r \le k_{t,N}$ are integers. 
$$ f^{(N)}_{j_0,\ldots, j_r}(\bar u) = \left(\prod\limits_{k=0}^{r-1} e^{-\lambda_{j_k}(s_{k+1}-s_k-1){T\over N}}\right)e^{-\lambda_{j_r}(k_{t,N}-s_r){T\over N}},\ \bar u \in C^{(N)}(s_1,\ldots, s_k),$$
and $f^{(N)}_{j_0,\ldots, j_r}(\bar u)=0$ otherwise. Clearly, 
$$\lim_{N\to \infty} f^{(N)}_{j_0,\ldots,j_r}(u_1,\ldots, u_r) = \left(\prod\limits_{k=0}^{r-1} e^{-\lambda_{j_k}(u_{k+1} - u_k)}\right)e^{-\lambda_{j_r}(t-u_r)},$$ where $u_0=0$. Since the expression under integrals in (\ref{fin_dim_lim}) is bounded we can exchange $\lim$ and integrals. By Lemma \ref{unif_conv_mark_chain} we can exchange $\lim$ and $\sum_{j_0,\ldots, j_r}$. Thus, we have
\begin{align*}
 \lim_{N\to \infty} &\mathbb{P}^{(N)}\left(Y^{(N)}_{\tau^{(N)}_r} = j_r, \tau^{(N)}_r \le t < \tau^{(N)}_{r+1}\ \middle|\ Y^{(N)}_0=j_0\right) \\
 &= \int_{u_1}^t \ldots \int_{u_{r-1}}^t  \lsum_{j_1,\ldots, j_{r-1}} \lim_{N\to \infty} \left(f^{(N)}_{j_0,\ldots,j_r}(\bar u) \prod\limits_{k=0}^{r-1} \left(p^{(N)}_{j_k j_{k+1}} {N\over T}\right)\right)d \bar u \\
 &= \int_{u_1}^t \ldots \int_{u_{r-1}}^t  \lsum_{j_1,\ldots, j_{r-1}} \left(\prod\limits_{k=0}^{r-1} e^{-\lambda_{j_k}(u_{k+1}-u_k)}\lambda_{j_k j_{k+1}}\right) e^{-\lambda_{j_r} (t-u_r)}d \bar u \\
 &= \prb\left(Y_{\tau_r} = j_r, \tau_r \le t < \tau_{r+1}\ \middle|\ Y_0=j_0\right),
\end{align*}
where $u_0=0, u_{k+1}=t$. Thus, we proved the base of induction. The step of induction now follows directly from the strong Markov property: 
\begin{align*}
 \prb^{(N)}&\left(Y^{(N)}_{\tau^{(N)}_i} = x_i, 0\le i\le n, B^{(N)}_r\right)\\
 &= \prb^{(N)}\left(Y^{(N)}_{\tau^{(N)}_i} = x_i, 0\le i\le n-1, B^{(N)}_{r-1}\right) \prb^{(N)}\left(Y^{(N)}_{\tau^{(N)}_1} = j_r\ \middle|\ Y^{(N)}_0 = j_{r-1}\right).
\end{align*}
The result now follows from the proven basis of induction and the assumption about the induction step. The theorem is proved.
\end{proof}

\begin{thm}\label{YN_conv}
  Processes $\left(Y^{(N)}_t\right)_{t\in [0,T]}$ converge in Skorokhod topology to $(Y_t)_{t\in [0,T]}$, as $N\to \infty $ .
\end{thm}
\begin{proof} In Theorem \ref{approx_y_thm2} we already proved the convergence of finite-dimensional distributions. Therefore, by Theorem 4, section VI.5, \cite{gikhman-skorokhod-stoch-proc},   we have to verify that for all $\delta > 0$
\begin{equation}\label{skorokhod_cond}
  \lim_{h\to 0}\limsup\limits_{N\to \infty} \sup_{x\in \N, 0\le s-t\le h}\prb^{(N)}\left( \left|Y^{(N)}_s - Y^{(N)}_t \right| > \delta\ \middle|\ Y^{(N)}_t = x\right) = 0.
\end{equation}

Let us examine the probability
$$\prb^{(N)}\left( \left|Y^{(N)}_{t+h} - Y^{(N)}_t \right| > \delta\ \middle|\ Y^{(N)}_t = x\right). $$
Since the chain $Y^{(N)}$ takes values in the set $\N$, the condition $\left| Y^{(N)}_{t+h} - Y^{(N)}_t \right| > \delta$ means that $Y^{(N)}_{t+h} \neq Y^{(N)}_t$. Thus, we can write
\begin{align*}
\prb^{(N)}&\left( \left|Y^{(N)}_{t+h} - Y^{(N)}_t \right| > \delta\ \middle|\ Y^{(N)}_t = x\right) = \prb^{(N)}\left( Y^{(N)}_{t+h} \neq x |\ Y^{(N)}_t = x\right)  \\
 &= \prb^{(N)}\left( Y^{(N)}_{h} \neq x |\ Y^{(N)}_0 = x\right).
\end{align*}
where the last equality follows from homogeneity. 
For a fixed $h$ and large enough N we have
\begin{equation*}
    \prb^{(N)}\left( Y^{(N)}_{h} \neq x |\ Y^{(N)}_0 = x\right) = 
    1-p^{(N)}_{xx}\left(\left\lfloor {hN\over T}\right\rfloor \right) \le 1 - \exp\left(-\lambda_x {T\over N}\left\lfloor {hN\over T}\right\rfloor \right),
\end{equation*}
where we used the inequality
$$ \prb^{(N)}\left(Y^{(N)}_k = x\ \middle|\ Y_0 = x\right) \ge \prb^{(N)}\left(Y^{(N)}_1 = Y^{(N)}_2 = \ldots = Y^{(N)}_k=x\ \middle|\ Y_0 = x\right). $$

Now we can rewrite the left hand side of (\ref{skorokhod_cond}) as
\begin{align*}
 \lim_{h\to 0} &\limsup\limits_{N\to \infty} \sup_{x\in \N, 0\le s-t\le h}\prb^{(N)}\left( \left|Y^{(N)}_s - Y^{(N)}_t \right| > \delta\ \middle|\ Y^{(N)}_t = x\right)\\
 &\le \lim_{h\to 0} \limsup\limits_{N\to \infty} \sup_{x\in \N, 0\le s-t\le h}  \left(1 - \exp\left(-\lambda_x {T\over N}\left\lfloor {(s-t)N\over T}\right\rfloor \right)\right) \\
 &\le 1 - \lim_{h\to 0} e^{-\lambda_* h} = 0.
\end{align*}

\end{proof}

\section{Weak convergence to a geometric Brownian motion with Markov
switching in the multiplicative scheme of series}\label{vol_conv}

Let us consider a sequence of independent random variables  $\left\{R^{(N)}_{i,j}, 0\le i \le N, j \in \N\right\}$ introduced in Section \ref{discrete-time-definitions}.
Assume that\\
(i) $R^{(N)}_{i,j} > -1$ a.s., and there exist non-negative real numbers $\{\gamma_N, N\ge 1\}$, such that $\sup_{i,j} \left|R^{(N)}_{i,j}\right| \le \gamma_N$, a.s. and $\gamma_N\to 0$,\ $N\to \infty$.\\
(ii) $\mu^{(N)}_j := \E R^{(N)}_{i,j}$ does not depend on $i$, and $\lim\limits_{N\to \infty} \left(1 + \mu^{(N)}_j\right)^N = e^{\mu_j T}.$  \\
(iii) For every fixed $t \in [0,T]$
\begin{equation}\label{unif_conv_of_variances} 
\sup_j \left|\lsum_{i=1}^{t^{(N)}} \Var R^{(N)}_{i,j} - \sigma^2_j t\right| \to 0,\ N\to \infty,
\end{equation}
where $t^{(N)} = \lfloor {Nt\over T}\rfloor.$\\
(iv) $\left\{R^{(N)}_{i,j}, 0\le i\le N, j\in \N\right\}$ and $\left\{Y^{(N)}_k, 1\le k\le N\right\}$ are independent.

Recall the process $\left(U^{(N)}_t\right)_{t\ge 0}$ defined in Section \ref{discrete-time-definitions}
\begin{equation}\label{tproc_def}
U^{(N)}_t = \lsum_{i=1}^{t^{(N)}} \log\left(1 + R^{(N)}_{i, Y^{(N)}_i}\right). 
\end{equation}
Note that the trajectories of this process are càdlàg.
The main result of the paper is stated in the following theorem.
\begin{thm}\label{main_result}
The following convergence takes place
$$ \left(U^{(N)}_t\right)_{t\in [0,T]} \to \left(\int_0^t \sigma_{Y_s} dW_s + \int_0^t \left(\mu_{Y_s} - {\sigma^2_{Y_s}\over 2}\right)ds \right)_{t\in [0,T]},$$
\end{thm}
in Skorokhod topology, when $ N\to \infty$.
\begin{proof}


We start our proof by establishing the convergence of finite-dimensional distributions.
To this end, we consider $\{\alpha_1, \ldots, \alpha_n\} \subset \R$, 
$ 0 = t_0 \le t_1 \le \cdots \le t_n \le T$, and show the convergence of the characteristic functions
\begin{align}\label{char_func_conv}
\E^{(N)} &\exp\left(i\lsum_{k=0}^{n-1} \alpha_{k+1} \left(U^{(N)}_{t_{k+1}} - U^{(N)}_{t_k}\right)\right)\\  \nonumber
&\to \E \exp\left(i\lsum_{k=0}^{n-1} \alpha_{k+1} \left\{\int_{t_k}^{t_{k+1}}\sigma_{Y_s} dW_s + \int_{t_k}^{t_{k+1}}\left(\mu_{Y_s}- { \sigma^2_{Y_s}\over 2}\right)ds\right\}\right),\ N\to \infty.
\end{align}
Let us introduce some notation:
\begin{gather*}
\xi^{(N)}_k = \exp\left(i\alpha_{k+1} \left(U^{(N)}_{t_{k+1}} - U^{(N)}_{t_k}\right)\right), 0\le k \le n-1,\\  
\xi_k = \exp\left(i\alpha_{k+1} \left\{\int_{t_k}^{t_{k+1}}\sigma_{Y_s}dW_s + \int_{t_k}^{t_{k+1}}\left(\mu_{Y_s}- { \sigma^2_{Y_s}\over 2}\right)ds\right\}\right), 0\le k \le n-1,  \\
\Xi^{(N)} = \prod\limits_{j=1}^n \xi^{(N)}_j, \Xi = \prod\limits_{j=1}^n \xi_j.
\end{gather*}
Thus, the left hand side of (\ref{char_func_conv}) equals $\mathbb{E}^{(N)}\left[\Xi^{(N)}\right]$ and the right hand side equals $\mathbb{E}\left[\Xi\right]$.

Denote by $r_{k} = N_{t_k}$ a number of jumps of the process $\left(Y^{(N)}_t\right)_{t\ge 0}$ before the time $t_k$, so that
$$ \tau_{r_k} \le t_k < \tau_{r_k + 1}.$$

This allows us to represent the interval $[t_{k}, t_{k+1}]$, $0\le k\le n-1$ as a union 
of renewal intervals in the following way
$$ [t_k, t_{k+1}] = \left[t_k, \tau^{(N)}_{r_k + 1}\right) \bigcup \left[\tau^{(N)}_{r_k+1}, \tau^{(N)}_{r_k + 2}\right)\bigcup \cdots \bigcup \left[\tau^{(N)}_{r_{k+1}}, t_{k+1}\right]. $$
For ease of notation we can renumerate these intervals as
$$ [t_k, t_{k+1}] = \left[s_{k,0}, s_{k,1}\right) \bigcup \left[s_{k,1}, s_{k,2}\right)\bigcup \cdots \bigcup \left[s_{k, m_k-1}, s_{k,m_k}\right], $$
where $s_{k,0} = t_k$, $m_k = r_{k+1} - r_k + 1$, $s_{k,m_k} = t_{k+1}$. In case $m_k > 1$ we have $s_{k,j} = \tau^{(N)}_{r_k + j -1}$, $1\le j \le m_k-1$. \\
Note that
\begin{equation}
\lsum_{k} m_k \le N_T,
\end{equation}
where $N_T$ is a number of jumps before $T$ defined in (\ref{Nt_def}), and the process $\left(Y^{(N)}_t\right)_{t\ge 0}$ is taking a single value on each of the intervals $[s_{k,j}, s_{k,j+1})$. Denote this value by $y_{k,j}$.
Now, we can write the following representation
\begin{align*}
\xi^{(N)}_k &= \exp\left(i\alpha_{k+1} \left(U^{(N)}_{t_{k+1}} - U^{(N)}_{t_k}\right)\right) = \exp\left(i\alpha_{k+1} \left(U^{(N)}_{s_{k,m_k}} - U^{(N)}_{s_{k,0}}\right)\right)\\
&= \prod\limits_{j=1}^{m_k}\exp\left(i\alpha_{k+1} \left(U^{(N)}_{s_j} - U^{(N)}_{s_{j-1}}\right)\right) = \prod\limits_{j=1}^{m_k} \zeta^{(N)}_{k,j},
\end{align*}
where
\begin{align*}
 \zeta^{(N)}_{k,j} = \exp\left(i\alpha_{k+1} \left(U^{(N)}_{s_j} - U^{(N)}_{s_{j-1}}\right)\right),\ 1 \le j \le m_k.
\end{align*}

This construction is illustrated below (we omitted upper indices $N$ for brevity).\\
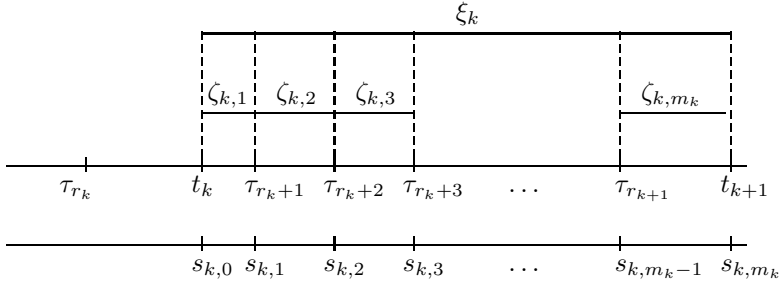
\begin{picture}(340,120)
\put(180,105){$\xi_k$}
\put(84,100){\line(1,0){200}}

\put(86,75){$\zeta_{k,1}$}
\put(84,70){\line(1,0){20}}

\put(112,75){$\zeta_{k,2}$}
\put(104,70){\line(1,0){30}}

\put(142,75){$\zeta_{k,3}$}
\put(134,70){\line(1,0){30}}

\put(250,75){$\zeta_{k,m_k}$}
\put(242,70){\line(1,0){40}}

\multiput(84,48)(0,5){11}{\line(0,1){3}} 
\multiput(104,48)(0,5){11}{\line(0,1){3}} 
\multiput(134,48)(0,5){11}{\line(0,1){3}} 
\multiput(164,48)(0,5){11}{\line(0,1){3}} 
\multiput(242,48)(0,5){11}{\line(0,1){3}} 
\multiput(284,48)(0,5){11}{\line(0,1){3}}

\put(10,50){\line(1,0){280}}
\put(40,48){\line(0,1){5}}
\put(30,40){$\tau_{r_k}$}
\put(84,48){\line(0,1){5}}
\put(80,40){$t_k$}
\put(104,48){\line(0,1){5}}
\put(100,40){$\tau_{r_k + 1}$}
\put(134,48){\line(0,1){5}}
\put(130,40){$\tau_{r_k + 2}$}
\put(164,48){\line(0,1){5}}
\put(160,40){$\tau_{r_k + 3}$}
\put(200,40){$\ldots$}

\put(242,48){\line(0,1){5}}
\put(240,40){$\tau_{r_{k+1}}$}
\put(284,48){\line(0,1){5}}
\put(280,40){$t_{k+1}$}

\put(10,20){\line(1,0){280}}
\put(84,18){\line(0,1){5}}
\put(80,10){$s_{k,0}$}
\put(104,18){\line(0,1){5}}
\put(100,10){$s_{k,1}$}
\put(134,18){\line(0,1){5}}
\put(130,10){$s_{k,2}$}
\put(164,18){\line(0,1){5}}
\put(160,10){$s_{k,3}$}
\put(200,10){$\ldots$}

\put(242,18){\line(0,1){5}}
\put(240,10){$s_{k,m_{k}-1}$}
\put(284,18){\line(0,1){5}}
\put(280,10){$s_{k,m_{k}}$}
\end{picture}

Note that the random variables $\zeta^{(N)}_{k,j}$ are conditionally independent given $\left(Y^{(N)}_t\right)_{t\in [0,T]}$.
More formally, let us denote $\sigma$-field generated by $\left(Y^{(N)}_t\right)_{t\in [0,T]}$ by 
\begin{equation}\label{Ysigma_field_def}
\EuScript{Y}^{(N)} = \sigma\left[Y^{(N)}_s,\ 0\le s\le T\right].
\end{equation}
Conditional independence then implies
\begin{align}\label{zeta_cond_ind}
\E^{(N)}&\left[\E^{(N)}\left[\prod\limits_{k,j}\zeta^{(N)}_{k,j} \middle|\ \EuScript{Y}^{(N)}\right]\mathbbm{1}_{N_T = m}\right] \\
    &= \E^{(N)}\left[\prod\limits_{k=0}^{n-1} \prod\limits_{j=1}^{m_k}\E^{(N)}\left[\zeta^{(N)}_{k,j} \middle|\ \EuScript{Y}^{(N)}\right]\mathbbm{1}_{N_T = m}\right]. \nonumber
\end{align}

We know from  Theorem 5.53,  \cite{felmer}  that for a fixed $j$ and any $0 \le s < t \le T$   the following 
weak convergence holds:
$$ \lsum_{k = s^{(N)}}^{t^{(N)}} \log\left(1 + R^{(N)}_{k,j}\right) \to \sigma_j (W_t - W_s) + \left(\mu_j - {\sigma^2_j \over 2}\right)(t-s), $$
so that for any $\alpha \in \R$
$$ \E^{(N)} \exp\left(i \alpha \lsum_{k = s^{(N)}}^{t^{(N)}} \log\left(1 + R^{(N)}_{k,j}\right) \right) \to \exp\left(-{\alpha^2\sigma^2_j(t-s)\over 2} + i\alpha \left(\mu_j - {\sigma^2_j\over 2}\right)(t-s)\right), $$
as $N\to \infty$, where $i^2=-1$. Denote the limit in the previous formula by
\begin{equation}\label{phi_jst_def}
 \varphi_{j,s,t}(\alpha) :=\exp\left(-{\alpha^2\sigma^2_j(t-s)\over 2} + i\alpha \left(\mu_j - {\sigma^2_j\over 2}\right)(t-s)\right).
\end{equation}

Recall that $Y^{(N)}_t\big|_{[s_{k,j}, s_{k,j+1})} = y_{k,j}$, whence 
$$ \zeta^{(N)}_{k,j} = \exp\left(i\alpha_{k+1} \left(U^{(N)}_{s_j} - U^{(N)}_{s_{j-1}}\right)\right)= \exp\left(i\alpha_{k+1} \lsum_{l=s^{(N)}_{j-1}}^{s^{(N)}_j}\log\left( 1 + R^{(N)}_{l,v }\right)\right),$$
where $v= y_{k,j-1}$.
Therefore
\begin{align*}
 \E^{(N)}&\left[\Xi^{(N)} \mathbbm{1}_{N_T = m}\right] =  \E^{(N)}\left[\E^{(N)}\left[\Xi^{(N)} \middle|\ \EuScript{Y}^{(N)}\right]\mathbbm{1}_{N_T = m}\right] \\
&=\E^{(N)}\left[\E^{(N)}\left[\prod\limits_{k,j}\zeta^{(N)}_{k,j} \middle|\ \EuScript{Y}^{(N)}\right]\mathbbm{1}_{N_T = m}\right]\\
    &= \E^{(N)}\left[\prod\limits_{k=0}^{n-1} \prod\limits_{j=1}^{m_k}\E^{(N)}\left[\zeta^{(N)}_{k,j} \middle|\ \EuScript{Y}^{(N)}\right]\mathbbm{1}_{N_T = m}\right],
\end{align*}    
where $\EuScript{Y}^{(N)}$ is defined by (\ref{Ysigma_field_def}) and we used conditional independence  (\ref{zeta_cond_ind}) in the last equality.   

Denote by
$$ \varphi_{k,j}= \varphi_{y_{k,j-1}, s_{k,j-1}, s_{k,j}}(\alpha_{k+1}),\ 1\le j\le m_k,$$
and note that for any $\varepsilon > 0$  by Lemma \ref{uniform_convergence_of_char_functions} we can choose $N$ so big that
\begin{equation}\label{ank_ineq} 
\left|\E^{(N)}\left[\zeta^{(N)}_{k,j} - \varphi_{k,j} \middle|\ \EuScript{Y}^{(N)}\right]\right| \le  {\varepsilon\over nm},
\end{equation}
a.s. for all $0\le k \le n-1$ and all $1\le j\le m_k$.

We will consequently add and subtract  $\varphi_{k,j}$ from each term in the product 
$$\prod\limits_{k,j}\E^{(N)}\left[\zeta^{(N)}_{k,j} \middle|\ \EuScript{Y}^{(N)}\right]$$ 
and take into account that
$$ \left| \E^{(N)}\left[\zeta^{(N)}_{k,j}\middle|\ \EuScript{Y}^{(N)} \right]\right| \le \E^{(N)}\left[\left|\exp\left(i\alpha_{k+1} (U^{(N)}_{s_{k,j}} - U^{(N)}_{s_{k, j-1}})\right)\right|\ \middle|\ \EuScript{Y}^{(N)}\right] = 1.$$
Thus, we have
\begin{align*}
 \prod\limits_{k=0}^{n-1}&\prod\limits_{j=1}^{m_k}\E^{(N)}\left[\zeta^{(N)}_{k,j} \middle|\ \EuScript{Y}^{(N)}\right] \\
 &= \left(\prod\limits_{k=0}^{n-1}\prod\limits_{j=1}^{m_k-1} \E^{(N)}\left[\zeta^{(N)}_{k,j} \middle|\ \EuScript{Y}^{(N)}\right]\right) \E^{(N)}\left[\varphi_{n-1,m_k}\middle|\ \EuScript{Y}^{(N)}\right] \\
 &+\left(\prod\limits_{k=0}^{n-1}\prod\limits_{j=1}^{m_k-1} \E^{(N)}\left[\zeta^{(N)}_{k,j} \middle|\ \EuScript{Y}^{(N)}\right]\right) \E^{(N)}\left[\zeta^{(N)}_{n-1,m_k} - \varphi_{n-1,m_k}\middle|\ \EuScript{Y}^{(N)}\right]\\
 &= \left(\prod\limits_{k=0}^{n-1}\prod\limits_{j=1}^{m_k-1} \E^{(N)}\left[\zeta^{(N)}_{k,j} \middle|\ \EuScript{Y}^{(N)}\right]\right) \E^{(N)}\left[\varphi_{n-1,m_k}\middle|\ \EuScript{Y}^{(N)}\right] + a_{n-1,m_k} \\
 &= \left(\prod\limits_{k=0}^{n-1}\prod\limits_{j=1}^{m_k-2} \E^{(N)}\left[\zeta^{(N)}_{k,j} \middle|\ \EuScript{Y}^{(N)}\right]\right) \prod\limits_{j=m_k-1}^{m_k}\E^{(N)}\left[\varphi_{n-1,j}\middle|\ \EuScript{Y}^{(N)}\right] \\
 &+ a_{n-1,m_k} + a_{n-1, m_k-1} = \prod\limits_{k,j}\E^{(N)}\left[\varphi_{k,j}\middle|\ \EuScript{Y}^{(N)}\right] + \lsum_{k,j} a_{k,j},
\end{align*}
where
\begin{align*}
a_{k,j} = \left(\prod\limits_{l=0}^k \prod_{p=1}^{j-1}\E^{(N)}\left[\zeta^{(N)}_{l,p} \middle|\ \EuScript{Y}^{(N)}\right]\right)\E^{(N)}\left[\zeta^{(N)}_{k,j} - \varphi_{k,j} \middle|\ \EuScript{Y}^{(N)}\right]\left(\prod\limits_{l\ge k}\prod\limits_{p > j}\E^{(N)}\left[\varphi_{l,p} \middle|\ \EuScript{Y}^{(N)}\right]\right).
\end{align*}
Using (\ref{ank_ineq}) we obtain
\begin{equation}\label{ank_ineq_2}
|a_{k,j}| \le {\varepsilon\over mn},\ \mbox{and}\ a := \lsum_{k,j} a_{k,j} \in (-\varepsilon, \varepsilon)
\end{equation}
a.s. on $\{N_T = m\}$.
From (\ref{ank_ineq_2}) we get
\begin{align*}
\E^{(N)}&\left[\Xi^{(N)} \mathbbm{1}_{N_T=m}\right] -a = \E^{(N)}\left[\prod\limits_{k,j}\E^{(N)}\left[\varphi_{k,j}\ \middle|\ \EuScript{Y}^{(N)} \right]\mathbbm{1}_{N_T = m}\right] \\
&=\E^{(N)}\left[\prod\limits_{k=0}^{n-1}\E^{(N)}\left[\prod\limits_{j=1}^{m_k}\varphi_{k,j}\ \middle|\ \EuScript{Y}^{(N)} \right]\mathbbm{1}_{N_T = m}\right].
\end{align*}
Using Lemma \ref{phi_conv} we find  $m$ and $N_0$ such that for all $N\ge N_0$
\begin{align*}
\left| \E^{(N)}\left[\Xi^{(N)} \right] - \E\left[\Xi\right] \right| &\le \left| \E^{(N)}\left[\Xi^{(N)}\mathbbm{1}_{N_T\le m} \right] - \E\left[\Xi \mathbbm{1}_{N_T \le m}\right] \right| \\
&+ \mathbb{P}^{(N)}(N_T > m) + \mathbb{P}(N_T > m) < \varepsilon,
\end{align*}
and thus convergence (\ref{char_func_conv}) follows.

To conclude the proof, we have to show that the family of probability distributions $\mathbb{P}^{(N)}$ is tight. This is an  immediate consequence of  Theorem 13.2 from \cite{billingsley}. Note  that conditions of this theorem are verified in Lemma \ref{tightness_conditions}.
\end{proof}

\section{Auxiliary lemmas}
\begin{lem}\label{unif_conv_mark_chain}
Let $\{p_{ij}(t), i,j\in \N, t\ge 0\}$ be transition probabilities for the Markov process $\left(Y_t\right)_{t\ge 0}$ with the infinitesimal matrix $\mathbb{A}$ given in (\ref{intensity_matrix}), for which
$$\lambda_* = \inf_i \lambda_i > 0,\ \mbox{and}\ \lambda^* = \sup_i \lambda_i < \infty.$$
Then the following series converge uniformly in $h$ for every fixed $i\in \N$ and $m>0$:
$$ S_{1,i,m}(h) = \lsum_{j_1\neq i} \lsum_{j_2\neq j_1} \ldots \lsum_{j_m \neq j_{m-1}} \prod\limits_{k=1}^m {p_{j_{k-1} j_k}(h)\over h}, $$
$$ S_{2,i,m}(h) = \lsum_{j_1\neq i} \lsum_{j_2\neq j_1} \ldots \lsum_{j_m \neq j_{m-1}} \prod\limits_{k=1}^m \left({p_{j_{k-1} j_k}(h)\over h}e^{-\lambda_{j_k} t_i}\right), $$
where $\{t_1,\ldots, t_m\} \subset [0,T]$. 
In particular,
\begin{align}\label{pij_lim_exchange}
\lim_{h\to 0} \lsum_{j_1\neq i} &\lsum_{j_2\neq j_1} \ldots \lsum_{j_m \neq j_{m-1}} \prod\limits_{k=1}^m \left({p_{j_{k-1} j_k}(h)\over h}e^{-\lambda_{j_k} t_i}\right)\\
&=  \lsum_{j_1\neq i} \lsum_{j_2\neq j_1} \ldots \lsum_{j_m \neq j_{m-1}} \prod\limits_{k=1}^m \left(\lambda_{j_{k-1}j_k} e^{-\lambda_{j_k} t_i}\right). \nonumber
\end{align}
\end{lem}
\begin{proof}
It is a well-known fact that for a Markov process generated by $\mathbb{A}$ we have
\begin{equation}\label{unif_conv_mark_chain_1}
\lim_{h\to 0} p_{ij}(h)/h = \lambda_{ij},
\end{equation}
see \cite{chung}, Sections II.2 and II.3, Theorem 1 from Section II.3 (p. 130) in particular. 
Let us consider:
\begin{align*}
\sup_i {1-p_{ii}(h)\over h} &\le \sup_i {p_{ii}(h) - e^{-\lambda_i h}\over h} + \sup_i {1-e^{-\lambda_i h}\over h} \\
&\le \sup_i {\mathbb{P}(\tau_1 < h | Y_0 = i)\over h} + \sup_i \lambda_i  = \sup_i {1-e^{-\lambda_i h}\over h} + \lambda^* \\
&\le {1-e^{-\lambda^* h}\over h } + \lambda^*  \le 2\lambda^*,
\end{align*}
where we used an elementary inequality $1 - e^{-x} \le x,\ x\ge 0$.
Let us prove that $S_{1,i,m}(h)$ is a uniformly convergent series. We establish this by induction in $m$. 
For $m=1$
\begin{equation*}
\lsum_{j\neq i} {p_{ij}(h)\over h} = \left\{
\begin{array}{l}
{1-p_{ii}(h)\over h},\ h \in (0,1],\\
\lambda_i,\ h=0,
\end{array}\right.
\end{equation*}
is a continuous function for $h\in [0,1]$. Since ${p_{ij}(h)\over h} \ge 0$, by Dini's theorem we have uniform convergence. \\
Assume now that the statement is true for $m-1$. We have then
\begin{align*}
  S_{1,i,m}(h) &= \lsum_{j_1\neq i} \lsum_{j_2\neq j_1} \ldots \lsum_{j_m \neq j_{m-1}} \prod\limits_{k=1}^m {p_{j_{k-1} j_k}(h)\over h} \\
  &= \lsum_{j_1\neq i} \lsum_{j_2\neq j_1} \ldots \lsum_{j_{m-1} \neq j_{m-2}} \left({1-p_{j_m j_m}(h)\over h}\prod\limits_{k=1}^m {p_{j_{k-1} j_k}(h)\over h}\right) \\
  \le 2\lambda^* S_{1,i,m-1}(h).
\end{align*}
Since $S_{1,i,m-1}(h)$ is uniformly convergent by assumption we derive that tails of $S_{1,i,m}$ converge to zero
\begin{align*} 
\lim_{n\to \infty} \sup_h &\lsum_{j_1>\in T^{(1)}_n} \lsum_{j_2 \in T^{(2)}_n} \ldots \lsum_{j_m \in T^{(m)}_n } \prod\limits_{k=1}^m {p_{j_{k-1} j_k}(h)\over h} \\
&\le 2\lambda^* \lim_{n\to \infty} \sup_h \lsum_{j_1 \in T^{(1)}_n} \lsum_{j_2 \in T^{(2)}_{n}} \ldots \lsum_{j_{m-1} \in T^{(m-1)}_n} \prod\limits_{k=1}^{m-1} {p_{j_{k-1} j_k}(h)\over h}  = 0,
\end{align*}
where at least one of the set $T^{(i)}_n$ consists only of elements larger or equal to $n$.
which concludes the proof of uniform convergence of $S_{1,i,m}$. 
Using the same reasoning we derive the uniform convergence of $S_{2,i,m}(h)$.
Formula (\ref{pij_lim_exchange}) follows from the uniform convergence of $S_{2,i,m}(h)$.
\end{proof}

It what follows we will use notations introduced in Section \ref{vol_conv}.
\begin{lem}\label{tightness_conditions}
Let $\left(U^{(N)}_t\right)_{t\in [0,T]}$ be a process defined at (\ref{tproc_def}). Then the following convergences hold
\begin{equation}\label{UNt_tight_cond_1}
\lim_{c\to \infty} \limsup_{n} \mathbb{P}^{(N)}\left(\sup_{t\in [0,T]} \left|U^{(N)}_t\right| \ge c\right) = 0,
\end{equation}
and for every $\varepsilon > 0$
\begin{equation}\label{UNt_tight_cond_2}
\lim_{\delta \to 0} \limsup_{n} \mathbb{P}^{(N)}\left( \sup_{\substack{t,s \in [0,T]\\ |t-s|<\delta }} \left|U^{(N)}_t - U^{(N)}_s\right| \ge \varepsilon \right) = 0.
\end{equation}
\end{lem}
\begin{proof}
First, let us note that $\sup_{N, t\in [0,T]}\E^{(N)} \left| U^{(N)}_t \right| < \infty $ (see \cite{felmer}, Section 5.7).
Thus, we can denote mean values by
$$ m^{(N)}_{ij} = \E^{(N)} \log\left(1 + R^{(N)}_{i,j}\right),\ \mbox{and}\ M^{(N)}_t = \E^{(N)} U^{(N)}_t, $$
and introduce a centered process
$$ \tilde U^{(N)}_t = U^{(N)}_t - M^{(N)}_t. $$
For a fixed $c>0$ we can produce the following upper bound:
\begin{align*}
 \mathbb{P}^{(N)}&\left(\sup_{t\in [0,T]} \left|\tilde U^{(N)}_t\right| \ge c\right) = \mathbb{P}^{(N)}\left(\sup_{0\le n \le N} \left|\lsum_{i=1}^n \left(\log\left(1 + R^{(N)}_{i, Y^{(N)}_i}\right) - m^{(N)}_{i, Y^{(N)}_i}\right)\right| \ge c\right)\\
 &=\E^{(N)}\left[\mathbb{P}^{(N)}\left(\sup_{0\le n \le N} \left|\lsum_{i=1}^n \left(\log\left(1 + R^{(N)}_{i, Y^{(N)}_i}\right)- m^{(N)}_{i, Y^{(N)}_i}\right)\right| \ge c\ \middle|\ \EuScript{Y}^{(N)}\right)\right] \\ 
 &\le {1\over c^2} \E^{(N)}\left[ \Var\left( \lsum_{i=1}^N \log\left(1 + R^{(N)}_{i,Y^{(N)}_i}\right)\ \middle| \EuScript{Y}^{(N)}\right)\right],
\end{align*}
where we used Kolmogorov inequality. Using condition (iii) for the random variables $R^{(N)}_{i,j}$ we can write 
\begin{align*}
\Var&\left( \lsum_{i=1}^N \log\left(1 + R^{(N)}_{i,Y^{(N)}_i}\right)\ \middle| \EuScript{Y}^{(N)}\right) = \lsum_{k=0}^{N_T} \Var\left( \lsum_{i=\tau^{(N)}_{k}}^{\tau^{(N)}_{k+1} \wedge N} \log\left(1 + R^{(N)}_{i, y_k}\right)\ \middle| \EuScript{Y}^{(N)}\right) \\
&\le \lsum_{k=0}^{N_T} \left| \Var\left( \lsum_{i=\tau^{(N)}_{k}}^{\tau^{(N)}_{k+1} \wedge N} \log\left(1 + R^{(N)}_{i, y_k}\right)\ \middle| \EuScript{Y}^{(N)}\right) - \sigma^2_{y_k} \theta^{(N)}_k \right| + \lsum_{k=0}^{N_T} \sigma^{2}_{y_k}\theta^{(N)}_k\\
&\le \zeta^{(N)} + T \sigma^2 ,
\end{align*}
a.s., where $y_k = Y^{(N)}_{\tau^{(N)}_k}$, $\sigma^2 = \sup_j \sigma^2_j$ defined in (\ref{sigma_def}) and
$$ \zeta^{(N)} = \lsum_{k=0}^{N_T} \left| \Var\left( \lsum_{i=\tau^{(N)}_{k}}^{\tau^{(N)}_{k+1} \wedge N} \log\left(1 + R^{(N)}_{i, y_k}\right)\ \middle| \EuScript{Y}^{(N)}\right) - \sigma^2_{y_k} \theta^{(N)}_k \right|.$$
Note that $\zeta^{(N)} = \zeta^{(N)}\left(\tau^{(N)}_i, y_i,\ i\ge 0\right)$.
For a given $\varepsilon > 0$ let us choose $m=m(\varepsilon)$ such that $\mathbb{P}(N_T > m) < \varepsilon$, and denote by $\tilde \zeta^{(N)} = \zeta^{(N)} \mathbbm{1}_{N_T \le m}$. Therefore we can write 
$$ \tilde\zeta^{(N)} = \tilde\zeta^{(N)}(t_1,\ldots, t_m, j_1, \ldots, j_m). $$
From (\ref{unif_conv_of_variances}) we have that
$$\sup_{\substack{t_1<\ldots < t_m \le T, \\ j_1,\ldots,j_m \in \N}} \tilde \zeta^{(N)}(t_1,\ldots, t_m, j_1,\ldots, j_m) \to 0, N\to \infty.$$

Thus, we have
\begin{align*}
 \mathbb{P}^{(N)}&\left(\sup_{t\in [0,T]} \left|\tilde U^{(N)}_t\right| \ge c\right) \le \mathbb{P}^{(N)}\left(N_T > m\right) + \E^{(N)}\left[\zeta^{(N)} \mathbbm{1}_{N_T \le m}\right] + {T \over c^2} \sigma^2 \\
 &\le \mathbb{P}^{(N)}\left(N_T > m\right) + \sup_{\substack{t_1<\ldots < t_m \le T, \\ j_1,\ldots,j_m \in \N}} \tilde \zeta^{(N)}(t_1,\ldots, t_m, j_1,\ldots, j_m) + {T \over c^2} \sigma^2.
\end{align*}
By selecting big enough $N$ and $m$ we can make $ \mathbb{P}^{(N)}\left(N_T > m\right)$ and \\$\sup_{\substack{t_1<\ldots < t_m \le T, \\ j_1,\ldots,j_m \in \N}} \tilde \zeta^{(N)}(t_1,\ldots, t_m, j_1,\ldots, j_m)$ arbitrarily small, so that for any $\varepsilon > 0$ we obtain the following upper bound:  
\begin{align*}
 \lim_{c\to \infty} \limsup_N \mathbb{P}^{(N)}&\left(\sup_{t\in [0,T]} \left|\tilde U^{(N)}_t\right| \ge c\right) \\
 &\le \varepsilon + \limsup_N \E^{(N)}\left[\zeta^{(N)} \mathbbm{1}_{N_T \le m}\right] + \lim_{c\to \infty} {T \over c^2} \sigma^2 = \varepsilon.
\end{align*}
From the relation  $\sup_{N,t} M^{(N)}_t < \infty$   we can deduce  that 
$$ \lim_{c\to \infty} \limsup_N \mathbb{P}^{(N)}\left(\sup_{t\in [0,T]} \left| U^{(N)}_t\right| \ge c\right) = 0, $$
which concludes the proof of (\ref{UNt_tight_cond_1}). \\

Now our goal is  to prove (\ref{UNt_tight_cond_2}). 
Let $\Delta > 0$  be a fixed number, and choose $m$, $\delta > 0$ and $N_0 = N_0(\delta)$ such that for all $N\ge N_0$
\begin{equation}\label{lem2_m_ineq}
\sup_N \mathbb{P}^{(N)}\left(N_T > m\right) < \Delta / 3,
\end{equation}
and
\begin{equation}\label{lem2_N_ineq}
 \mathbb{P}^{(N)}\left(\min_i \theta^{(N)}_i > \delta, N_T \le m\right) \ge 1 - \Delta/3.
\end{equation}
The latter is possible because
\begin{align*}
\mathbb{P}^{(N)}&\left(\min_i \theta^{(N)}_i > \delta, N_T \le m\right) = \lsum_{k=0}^m \mathbb{P}^{(N)}\left(\min_{0\le i \le k} \theta^{(N)}_i > \delta, N_T=k\right)\\
&= \lsum_{k=0}^m\mathbb{P}^{(N)}\left(\theta^{(N)}_i > \delta, 0\le i\le k,\ \lsum_{i=0}^k \theta^{(N)}_i < T \le \lsum_{i=1}^{k+1}\theta^{(N)}_i \right)\\
&\to \mathbb{P}^{(N)}(N_T \le m),\ \delta \to 0,
\end{align*}
and $\mathbb{P}^{(N)}(N_T \le m) \to 1,\ m\to \infty$.
Then we can write
\begin{align*}
\mathbb{P}^{(N)}&\left( \sup_{\substack{t,s \in [0,T]\\ |t-s|<\delta }} \left|U^{(N)}_t - U^{(N)}_s\right| \ge \varepsilon \right) \le \mathbb{P}^{(N)}\left(N_T > m\right) + \mathbb{P}^{(N)}\left(N_T \le m, \min_i \theta^{(N)}_i \le \delta\right)  \\
&+ \E^{(N)}\left[\mathbb{P}^{(N)}\left( \sup_{\substack{t,s \in [0,T]\\ |t-s|<\delta }} \left|U^{(N)}_t - U^{(N)}_s\right| \ge \varepsilon \middle|\ \EuScript{Y}^{(N)}\right)\mathbbm{1}_{N_T \le m}\mathbbm{1}_{\min_i \theta^{(N)}_i > \delta}\right] \\
&\le {2\Delta \over 3} + \E^{(N)}\left[\mathbb{P}^{(N)}\left( \sup_{\substack{t,s \in [0,T]\\ |t-s|<\delta }} \left|U^{(N)}_t - U^{(N)}_s\right| \ge \varepsilon \middle|\ \EuScript{Y}^{(N)}\right)\mathbbm{1}_{B}\right],
\end{align*}
where $B = \left\{N_T \le m,\ \min_i \theta^{(N)}_i > \delta\right\}$. Note  that when a trajectory of $\left(Y^{(N)}_t\right)_{t\in [0,T]}$ belongs to $B$, there are only two possible cases in which there exist $t,s\in [0,T]$ such that $|t-s| < \delta$ and $\left|U^{(N)}_t - U^{(N)}_s\right| \ge \varepsilon$:\\
(a) when there are no jumps occurring  between $s$ and $t$, so that there exists  $j$   such that $\tau^{(N)}_j \le s<t < \tau^{(N)}_{j+1}$, or\\
(b) there is exactly one jump between $s$ and $t$. \\
Let us consider such $\omega \in B$ that
$$ \sup_{\substack{t,s \in [0,T]\\ |t-s|<\delta }} \left|U^{(N)}_t(\omega) - U^{(N)}_s(\omega)\right| \ge \varepsilon. $$
If option (a) is valid, then  
 $$\max_i \sup_{\substack{t,s \in \left[\tau^{(N)}_{i-1}(\omega), \tau^{(N)}_i(\omega)\right)\\ |t-s|<\delta }}\left|U^{(N)}_t(\omega) - U^{(N)}_s(\omega)\right| \ge \varepsilon, $$
otherwise there exists $i$  such that
$$ \sup_{\substack{s< \tau^{(N)}_i(\omega) < t\\ t-s<\delta}} \left|U^{(N)}_t(\omega) - U^{(N)}_s(\omega)\right| \ge \varepsilon. $$
Note  that, due to the condition (i),  we can select $N$ so big  that \\ $\sup_t\left|U^{(N)}_{t} - U^{(N)}_{t-}\right| < \varepsilon/3.$
Thus, we can conclude that either:
$$\sup_{\substack{t,s \in \left[\tau^{(N)}_{i-1}(\omega), \tau^{(N)}_i(\omega)\right)\\ |t-s|<\delta }}\left|U^{(N)}_t(\omega) - U^{(N)}_s(\omega)\right| \ge \varepsilon/3, $$
or
$$\sup_{\substack{t,s \in \left[\tau^{(N)}_{i}(\omega), \tau^{(N)}_{i+1}(\omega)\right)\\ |t-s|<\delta }}\left|U^{(N)}_t(\omega) - U^{(N)}_s(\omega)\right| \ge \varepsilon/3. $$
With this in mind we can write 
\begin{align*}
\E^{(N)}&\left[\mathbb{P}^{(N)}\left( \sup_{\substack{t,s \in [0,T]\\ |t-s|<\delta }} \left|U^{(N)}_t - U^{(N)}_s\right| \ge \varepsilon \middle|\ \EuScript{Y}^{(N)}\right)\mathbbm{1}_{B}\right] \\
&\le 2\E^{(N)}\left[\mathbb{P}^{(N)}\left( \max_{1\le i\le m}\sup_{\substack{t,s \in \left[\tau^{(N)}_{i-1}, \tau^{(N)}_i\right)\\ |t-s|<\delta }} \left|U^{(N)}_t - U^{(N)}_s\right| \ge {\varepsilon\over 3}   \middle|\ \EuScript{Y}^{(N)}\right)\mathbbm{1}_{B}\right] \\
&\le 2\lsum_{i=1}^m\E^{(N)}\left[\mathbb{P}^{(N)}\left(\sup_{\substack{t,s \in \left[\tau^{(N)}_{i-1}, \tau^{(N)}_i\right)\\ |t-s|<\delta }} \left|U^{(N)}_t - U^{(N)}_s\right| \ge {\varepsilon\over 3}   \middle|\ \EuScript{Y}^{(N)}\right)\mathbbm{1}_{B}\right].
\end{align*}
If trajectory of $\left(Y^{(N)}_t\right)_{t\in [0,T]}$  is fixed, we conclude that
$$ \left(U^{(N)}_t - U^{(N)}_{u}\right)_{t\in \left[u,v\right)} \to \left(\left(\mu_j - {\sigma^2_j\over 2}\right)\left(t - u\right)  + \sigma_{j}\left(W_t - W_u\right)\right)_{t\in [u,v)},$$
in Skorokhod topology, where $(W_t)_{t\ge 0}$ is a Wiener process, $u=\tau^{(N)}_{i-1}$, $v = \tau^{(N)}_i$, $j=Y^{(N)}_{\tau^{(N)}_{i-1}}$.
Denote by $\omega_{\delta, i}(U)$ the modulus of continuity of the process $\left(U^{(N)}_t\right)_{t\ge 0}$ on a time interval between
$(i-1)$th and $i$th jumps, so that
$$\omega_{\delta, i,j}\left(U^{(N)}\right) = \sup_{\substack{t,s \in \left[\tau^{(N)}_{i-1}, \tau^{(N)}_i\right)\\ |t-s|<\delta }} \left|U^{(N)}_t - U^{(N)}_s\right| = \sup_{\substack{t,s \in \left[\tau^{(N)}_{i-1}, \tau^{(N)}_i\right)\\ |t-s|<\delta }} \left|\lsum_{k=s^{(N)}}^{t^{(N)}} \log\left( 1 + R^{(N)}_{kj}\right)\right|.$$
Let us also denote by $\omega_{\delta, i,j}(Z)$ the modulus of continuity of the process 
$$ Z_t = \left(\mu_j - {\sigma^2_j\over 2}\right)\left(t - u\right)  + \sigma_{j}\left(W_t - W_u\right),$$
on the interval $[u,v)$, where $u=\tau^{(N)}_{i-1}$, $v=\tau^{(N)}_i$, and let $\omega_{\delta, i}(W)$ be a respective modulus of continuity of a Wiener process:  
$$ \omega_{\delta, i}(W) = \sup_{\substack{t,s \in \left[\tau^{(N)}_{i-1}, \tau^{(N)}_i\right)\\ |t-s|<\delta }} \left|W_t - W_s\right|.$$
Note that $\{\tau^{(N)}_k,\ k\ge 0\}$ are independent of $\left(W_t\right)_{t\ge 0}$.
Since modulus of continuity is a continuous function in Skorokhod topology, and the processes   $\left(U^{(N)}_t - U^{(N)}_{u}\right)_{t\in \left[u,v\right)}$ converge weakly in Skorokhod topology to $ (Z_t)_{t\in [u,v)}$, 
we have
$$ \lim_{N\to \infty} \E^{(N)} \left[\omega_{\delta, i,j}\left(U^{(N)}\right)\right] =  \E \left[\omega_{\delta, i,j}\left(Z\right)\right] \le \sup_j\left( \left(\mu_j - \sigma^2_j\over 2\right)\delta + \sigma_j \E \omega_{\delta, i}(W)\right).$$
With this in mind we can write
\begin{align*}
\mathbb{P}^{(N)}&\left(\sup_{\substack{t,s \in \left[\tau^{(N)}_{i-1}, \tau^{(N)}_i\right)\\ |t-s|<\delta }} \left|U^{(N)}_t - U^{(N)}_s\right| \ge {\varepsilon\over 3}   \middle|\ \EuScript{Y}^{(N)}\right) \le {3\over \varepsilon} \E^{(N)}\left[\omega_{\delta, i}\left(U^{(N)}\right)\middle|\ \EuScript{Y}^{(N)}\right] \\
&\to {3\over \varepsilon} \E[\omega_{\delta,i}(Z)] \le \sup_j {3\over \varepsilon}\left(\left(\mu_j - \sigma^2_j\over 2\right)\delta + \sigma_j \E\omega_{\delta, i}(W)\right)\\
&\le 2{\mu + \sigma\over 2\varepsilon} \delta + {3\sigma\over \varepsilon} \E\omega_{\delta,i}(W) \le 2{\mu + \sigma\over 3\varepsilon} \delta + {3\sigma\over \varepsilon} \E\omega_{\delta,[0,T]}(W),
\end{align*}
where convergence holds as $N\to \infty$, and
$$ \omega_{\delta,[0,T]}(W) = \sup_{\substack{t,s\in [0,T]\\ |t-s|<\delta}} |W_t - W_s|$$
is a modulus of continuity of a Wiener process on the interval $[0,T]$ which  does not depend on $i$. 
It is well-known that the modulus of continuity of the Wiener process satisfies the limit relation 
$$\lim_{\delta \to 0} \E\omega_{\delta,[0,T]}(W) = 0.$$
Recall that we fixed  $\Delta > 0$ and chose  $m$ such that
$$ \sup_N \mathbb{P}^{(N)}(N_T > m) < \Delta/3.$$
Let us now choose  $\delta$ so small that
$$ 4m\left(3{\mu + \sigma\over 2\varepsilon} \delta + {3\sigma\over \varepsilon} \E\omega_{\delta,[0,T]}(W)\right) < \Delta/3. $$
Finally, we can choose $N_1 = N_1(\delta) \ge N_0(\delta)$ such that for all $N> N_1$
$$ \mathbb{P}^{(N)}\left(\sup_{\substack{t,s \in \left[\tau^{(N)}_{i-1}, \tau^{(N)}_i\right)\\ |t-s|<\delta }} \left|U^{(N)}_t - U^{(N)}_s\right| \ge {\varepsilon\over 3}   \middle|\ \EuScript{Y}^{(N)}\right) \le {\Delta\over 6m}, $$
a.s., so that
\begin{align*}
\E^{(N)}&\left[\mathbb{P}^{(N)}\left( \sup_{\substack{t,s \in [0,T]\\ |t-s|<\delta }} \left|U^{(N)}_t - U^{(N)}_s\right| \ge \varepsilon \middle|\ \EuScript{Y}^{(N)}\right)\mathbbm{1}_{B}\right] \\
&\le 2\lsum_{i=1}^m\E^{(N)}\left[\mathbb{P}^{(N)}\left(\sup_{\substack{t,s \in \left[\tau^{(N)}_{i-1}, \tau^{(N)}_i\right)\\ |t-s|<\delta }} \left|U^{(N)}_t - U^{(N)}_s\right| \ge {\varepsilon\over 3}   \middle|\ \EuScript{Y}^{(N)}\right)\mathbbm{1}_{B}\right]\\
&\le {2m\Delta \over 6m} = {\Delta \over 3}.
\end{align*}
Finally, we proved that for any $\Delta > 0$ there exists $\delta > 0$ and $N_1=N_1(\delta)$ such that for all $N\ge N_1$ it holds that
\begin{align*}
\mathbb{P}^{(N)}&\left( \sup_{\substack{t,s \in [0,T]\\ |t-s|<\delta }} \left|U^{(N)}_t - U^{(N)}_s\right| \ge \varepsilon \right) \le \Delta,
\end{align*}
which completes the proof of the lemma.
\end{proof}

\begin{lem}\label{uniform_convergence_of_char_functions}
     In the notation introduced in section \ref{vol_conv}, we have the following convergence:
     $$ \sup_{s<t, j\in \N} \Delta^{(N)}_j(s,t) \to 0, N\to \infty,$$
     for any $\alpha \in \R$, where
     $$ \Delta^{(N)}_j(s,t) = \left|\E^{(N)} \exp\left(i \alpha \lsum_{k = s^{(N)}}^{t^{(N)}} \log\left(1 + R^{(N)}_{k,j}\right) \right) - \varphi_{j,s,t}(\alpha)\right|.$$
 \end{lem}
 \begin{proof}
 Recall the notation (\ref{phi_jst_def}) for  $\varphi_{j,s,t}$: 
 \begin{equation*}
   \varphi_{j,s,t}(\alpha) :=\exp\left(-{\alpha^2\sigma^2_j(t-s)\over 2} + i\alpha \left(\mu_j - {\sigma^2_j\over 2}\right)(t-s)\right).
 \end{equation*}
From \cite{felmer}, Theorem 5.53, we know that the random variables 
$$ S^{(N)}_j(s,t) = \lsum_{k=s^{(N)}}^{t^{(N)}} \log\left(1 + R^{(N)}_{k,j}\right)$$
weakly converge to $Z_j(s,t) = \sigma_j (W_t - W_s) + \left(\mu_j - {\sigma^2_j\over 2}\right)(t-s)$, where $\left(W_t\right)_{t\ge 0}$ is a Wiener process. Note that $\varphi_{j,s,t}(\alpha)$ is a characteristic function of $Z_j(s,t)$. Let us denote the characteristic function of $S^{(N)}_j(s,t)$ by 
$$ \varphi^{(N)}_{j,s,t}(\alpha) = \E^{(N)} \exp\left(i \alpha \lsum_{k = s^{(N)}}^{t^{(N)}} \log\left(1 + R^{(N)}_{k,j}\right) \right). $$
Since 
$$ \Delta^{(N)}_j(s,t) = \left| \varphi^{(N)}_{j,s,t}(\alpha) - \varphi_{j,s,t}(\alpha) \right|,$$
the convergence $\Delta^{(N)}_j(s,t) \to 0,\ N\to \infty$ is a direct implication of the weak convergence $S^{(N)}_j(s,t) \to^W Z_j(s,t)$. 
Let us show that convergence is uniform in $s,t,j$. 
For doing so, we will need an estimate for the rate of convergence in Central Limit Theorem. In particular, we know from \cite{chung}  (Lemma 5, Section 7.4,  p. 240)   that for any  $\alpha < 1/(4 \Gamma^{(N)}_j(s,t))$ 
$$ \Delta^{(N)}_j(s,t) \le C(\alpha) \Gamma^{(N)}_j(s,t), $$
where 
$$ \Gamma^{(N)}_j(s,t) = \E\left|\lsum_{k = s^{(N)}}^{t^{(N)}} \log\left(1 + R^{(N)}_{k,j}\right)\right|^3.$$
Applying  the  proof of Theorem 5.53, \cite{felmer}, p. 252,   we obtain  that $\Gamma^{(N)}_j(s,t)$ can be estimated as
$$ \Gamma^{(N)}_j(s,t) \le |\gamma_N| C \sigma^2_j (t-s) \le |\gamma_N| C_1 \to 0, N\to \infty,$$
where $\{\gamma_N, N\ge 1\}$ is a dominating sequence taken from condition (i), and $C_1=C_1(\mu, \sigma^2, T)$  is a constant.
Thus, we have shown that for every fixed $\alpha \in \R$
$$ \sup_{s<t, j\in \N} \Delta^{(N)}_j(s,t) \le |\gamma_N| C_2 \to 0,\ N\to \infty, $$
where $C_1 = C_1(\mu, \sigma^2, T, \alpha)$.
\end{proof}

\begin{lem}\label{phi_conv}
In terms of  the notation from Section \ref{vol_conv} we can state that  
for any $\varepsilon > 0$ there exists  $m$ and $N_0$ such  that for all $N\ge N_0$:
$$ \left|\E^{(N)}\left[\prod\limits_{k=0}^{n-1}\E^{(N)}\left[\prod\limits_{j=1}^{m_k}\varphi_{k,j}\ \middle|\ \EuScript{Y}^{(N)} \right]\mathbbm{1}_{N_T \le m}\right] -  \E\left[\Xi \mathbbm{1}_{N_t \le m}\right]\right| <\varepsilon. $$

\end{lem}
\begin{proof}
Let $\{\alpha_1, \ldots, \alpha_n\} \subset \R$ and $0=t_0 \le t_1< \ldots t_n \in \R$ be the numbers selected in Section \ref{vol_conv}, and let us define functions $\mu(x)$ and $\sigma(x)$ as linear interpolations of 
$\mu_k$ and $\sigma_k$. So, 
$$\mu(k + t) = \mu_k(1-t) + \mu_{k+1} t,\ t\in [0,1],\ k\in \N ,$$
$$\sigma(k + t) = \sigma_k(1-t) + \sigma_{k+1} t,\ t\in [0,1], k\in \N.$$
Let us also define functional $h \colon D([0,T]) \to \R$ as follows:
$$ h(x) = \exp\left(\int_0^T \left(- {\alpha^2(t)\sigma^2(x(t))\over 2} + i\alpha(t) \left(\mu(x(t)) - { \sigma^2(x(t))\over 2}\right)\right) dt\right),$$
where $\alpha(t) = \lsum_{k=1}^n \alpha_k \mathbbm{1}_{[t_{k-1}, t_k)}(t)$.

Functional $h$ is clearly  bounded (with $\sup_x |h(x)| \le 1$) and continuous functional on $D([0,T])$.
Therefore, we can use Theorem \ref{YN_conv} and obtain the following convergence:
$$ \E^{(N)}\left[h\left(\left(Y^{(N)}_t\right)_{t\in [0,T]}\right)\right] \to \E\left[h\left(\left(Y_t\right)_{t\in [0,T]}\right)\right],\ N\to \infty. $$
Note that,
$$ \prod\limits_{k=0}^{n-1}\E^{(N)}\left[\prod\limits_{j=1}^{m_k}\varphi_{k,j}\ \middle|\ \EuScript{Y}^{(N)} \right]\mathbbm{1}_{N_T = m} = h\left(\left(Y^{(N)}_t\right)_{t\in [0,T]}\right)\mathbbm{1}_{N_T=m}, $$
and 
$$ \Xi = h\left(\left(Y_t\right)_{t\in [0,T]}\right).$$
Finally, for any $\varepsilon > 0$ we can find such $N_0$ that for all $N\ge N_0$ the following probabilities can be estimated as 
$$ \mathbb{P}^{(N)}(N_T > m) + \mathbb{P}(N_T > m) < \varepsilon/2, $$
and
$$\left| \E^{(N)}\left[h\left(\left(Y^{(N)}_t\right)_{t\in [0,T]}\right)\right] -\E\left[h\left(\left(Y_t\right)_{t\in [0,T]}\right)\right]\right| \le \varepsilon /2, $$
so that 
\begin{align*}
 &\left|\E^{(N)}\left[\prod\limits_{k=0}^{n-1}\E^{(N)}\left[\prod\limits_{j=1}^{m_k}\varphi_{k,j}\ \middle|\ \EuScript{Y}^{(N)} \right]\mathbbm{1}_{N_T \le m}\right] - \E[\Xi\mathbbm{1}_{N_T\le m}]\right| \\
 &= \left| \E^{(N)}\left[\left(\left(Y^{(N)}_t\right)_{t\in [0,T]}\right)\mathbbm{1}_{N_T\le m}\right] - \E\left[[h\left(\left(Y_t\right)_{t\in [0,T]}\right)\mathbbm{1}_{N_T \le m} \right]\right| \\
&\le \left| \E^{(N)}\left[h\left(\left(Y^{(N)}_t\right)_{t\in [0,T]}\right)\right] -\E\left[h\left(\left(Y_t\right)_{t\in [0,T]}\right)\right]\right| \\
&+  \mathbb{P}^{(N)}(N_T > m) + \mathbb{P}(N_T > m) \le \varepsilon.
\end{align*}
\end{proof}


\section*{Acknowledgement}
YM is supported by The Swedish Foundation for Strategic Research, grant UKR24-0004,   by the Japan
Science and Technology Agency CREST, project reference number JPMJCR2115   and    the ToppForsk project no. 274410 of the Research Council of Norway with the title STORM: Stochastics for Time-Space Risk Models.

\input{references}

\end{document}

%% file: references.tex
%
%
%

%% file: golomoziy_kladivko_mishura.bbl
\begin{thebibliography}{99.}%

\bibitem{driftSwitching} Golomoziy, V., Mishura, Y., Kladivko, K.: A discrete-time model that weakly converges to a continuous-time geometric Brownian motion with Markov switching drift rate, Frontiers in Applied Mathematics and Statistics (2024)

\bibitem{billingsley} Billingsley, P.: Convergence of Probability Measures, ed. 2 (John Wiley and Sons, Inc.) (1999)

\bibitem{chung} Chung, K.L.: Markov chains with stationary transition probabilities, (Springer-Verlag) (1960)

\bibitem{gikhman-skorokhod-stoch-proc} Gikhman, I.I., Skorokhod, A.V.: The Theory of Stochastic Processes I, (Springer) (2004)

\bibitem{felmer} Föllmer, H., Schied, A.: Stochastic Finance: An Introduction in Discrete Time, ed. 3 (De Gruyter) (2011)

\bibitem{hubalek} Hubalek, F., Schachermayer, W.: When Does Convergence of Asset Price Processes Imply Convergence of Option Prices? Mathematical Finance, \textbf{8}(4), 385--403 (1995)

\bibitem{MiRa} Mishura, Y., Ralchenko, K.: Discrete-Time Approximations and Limit Theorems In Applications to Financial Markets, vol 2 (De Gruyter), 390 (2021)

\bibitem{prigent} Prigent, J.L.: Weak Convergence of Financial Markets, (Springer Berlin, Heidelberg) (2003)

\bibitem{savku_2017} Savku, E.: Advances in optimal control of markov regime-switching models with applications in finance and economics, (Middle East Technical University) (2017)

\bibitem{Savku2021} Savku, E., Weber, G.W.: A Regime-Switching Model with Applications to Finance: Markovian and Non-Markovian Cases. In: Haunschmied, J.L., Kovacevic, R.M., Semmler, W., Veliov, V. (eds.) Dynamic Economic Problems with Regime Switches, 287--309 (Springer International Publishing) (2021)



\bibitem{naik93} Naik, V.: Option Valuation and Hedging Strategies with Jumps in the Volatility of Asset Returns. Journal of Finance, \textbf{48}(5), 1969--1984 (1993)

\bibitem{dimasi95} Di Masi, G. B., Kabanov, Y. M., Runggaldier, W. J.: Mean-Variance Hedging of Options on Stocks with Markov Volatilities. Theory of Probability and its Applications, \textbf{39}(1), 172--182 (1995)

\bibitem{guo01} Guo, X.: Information and Option Pricings. Quantitative Finance, \textbf{1}(1), 38--44 (2001)

\bibitem{buffington02} Buffington, J., Elliott, R. J.: American Options with Regime Switching. International Journal of Theoretical and Applied Finance, \textbf{5}(5), 497--514 (2002)

\bibitem{guo04} Guo, X., Zhang, Q.: Closed-Form Solutions for Perpetual American Put Options with Regime Switching. SIAM Journal on Applied Mathematics, \textbf{64}(6), 2034--2049 (2004)

\bibitem{jobert06} Jobert, A., Rogers, L. C. G.: Option pricing with Markov-modulated dynamics. SIAM Journal on Control and Optimization, \textbf{44}(6), 2063--2078 (2006)

\bibitem{boyle07} Boyle, P., Draviam, T.: Pricing Exotic Options under Regime Switching. Insurance: Mathematics and Economics, \textbf{40}(2), 267--282 (2007)

\bibitem{kirkby20} Kirkby, J. L., Nguyen, D.: Efficient Asian option pricing under regime switching jump diffusions and stochastic volatility models. Annals of Finance, \textbf{16}(3), 307--351 (2020)

\bibitem{dimitrov21} Dimitrov, M., Jin, L., Ni, Y.: Properties of American options under a Markovian regime switching model. Communications in Statistics: Case Studies, Data Analysis and Applications, \textbf{7}(4), 573--589 (2021)


\bibitem{bollen98} Bollen, N. P. B.: Valuing options in regime-switching models. Journal of Derivatives, \textbf{6}(1), 38--50 (1998)

\bibitem{aingworth06} Aingworth, D. D., Das, S. R., Motwani, R.: A simple approach for pricing equity options with Markov switching state variables. Quantitative Finance, \textbf{6}(2), 95--105 (2006)

\bibitem{khaliq09} Khaliq, A. Q. M., Liu, R. H.: New numerical scheme for pricing American option with regime-switching. International Journal of Theoretical and Applied Finance, \textbf{12}(3), 319--340 (2009)

\bibitem{liu10} Liu, R. H.: Regime-switching recombining tree for option pricing. International Journal of Theoretical and Applied Finance, \textbf{13}(3), 479--499 (2010)

\bibitem{yuen10} Yuen, F. L., Yang, H.: Option pricing with regime switching by trinomial tree method. Journal of Computational and Applied Mathematics, \textbf{233}(8), 1821--1833 (2010)


\bibitem{yin04} Yin, G., Zhou, X. Y.: Markowitz's mean-variance portfolio selection with regime switching: From discrete-time models to their continuous-time limits. IEEE Transactions on Automatic Control, \textbf{49}(3), 349--360 (2004)

\bibitem{ma15} Ma, J., Zhu, T.: Convergence rates of trinomial tree methods for option pricing under regime-switching models. Applied Mathematics Letters, \textbf{39}, 13--18 (2015)

\bibitem{leduc17} Leduc, G., Zeng, X.: Convergence rate of regime-switching trees. Journal of Computational and Applied Mathematics, \textbf{319}, 56--76 (2017)



\end{thebibliography}
